\documentclass[leqno,11pt,a4paper]{amsart}
\pdfoutput=1

\usepackage[full]{textcomp}
\usepackage{newpxtext}
\usepackage{diagbox}
\usepackage{comment}
\usepackage{cabin}
\usepackage[varqu,varl]{inconsolata}
\usepackage[bigdelims,vvarbb]{newpxmath}
\usepackage[cal=cm,bb=esstix,bbscaled=1.05]{mathalfa}
\usepackage[margin=2.395cm]{geometry}
\usepackage{booktabs}
\usepackage{savesym}
\let\savedegree\bigtimes
\let\bigtimes\relax
\usepackage{mathabx}
\usepackage{amsmath}
\usepackage{graphicx}
\let\bigtimes\savedegree

\usepackage[dvipsnames]{xcolor}
\usepackage{tikz}
\usepackage{tikz-cd}
\usepackage{cite}
\usepackage{enumitem}
\usepackage{amsfonts}
\setlist[enumerate]{labelsep=*, leftmargin=1.5pc}
\setlist[enumerate]{label=\normalfont(\roman*), ref=\roman*}

\usepackage{aliascnt}

\newtheorem{thm}{Theorem}[section]

\newaliascnt{lemma}{thm}
\newtheorem{lemma}[lemma]{Lemma}
\aliascntresetthe{lemma}

\newaliascnt{cor}{thm}
\newtheorem{cor}[cor]{Corollary}
\aliascntresetthe{cor}

\newaliascnt{prop}{thm}
\newtheorem{prop}[prop]{Proposition}
\aliascntresetthe{prop}

\theoremstyle{definition}

\newaliascnt{example}{thm}
\newtheorem{example}[example]{Example}
\aliascntresetthe{example}

\newaliascnt{notation}{thm}
\newtheorem{notation}[notation]{Notation}
\aliascntresetthe{notation}

\newaliascnt{question}{thm}
\newtheorem{question}[question]{Question}
\aliascntresetthe{question}

\newaliascnt{conj}{thm}
\newtheorem{conj}[conj]{Conjecture}
\aliascntresetthe{conj}

\newaliascnt{remark}{thm}
\newtheorem{remark}[remark]{Remark}
\aliascntresetthe{remark}

\newaliascnt{definition}{thm}
\newtheorem{definition}[definition]{Definition}
\aliascntresetthe{definition}

\newtheorem*{definition*}{Definition}
\newtheorem*{notation*}{Notation}
\newtheorem*{organization*}{Organization}
\newtheorem*{theorem*}{Theorem}

\numberwithin{equation}{section}

\usepackage{hyperref}

\hypersetup{
    colorlinks=true,
    linkcolor={blue!90!black},
    citecolor={red!80!black},
    urlcolor={blue!50!black}
}

\usepackage[capitalise,nameinlink,noabbrev]{cleveref}

\crefname{thm}{Theorem}{Theorems}
\crefname{lemma}{Lemma}{Lemmas}
\crefname{cor}{Corollary}{Corollaries}
\crefname{prop}{Proposition}{Propositions}
\crefname{conj}{Conjecture}{Conjectures}
\crefname{definition}{Definition}{Definitions}
\crefname{remark}{Remark}{Remarks}
\crefname{example}{Example}{Examples}
\crefname{notation}{Notation}{Notations}

\crefname{innercustomthm}{Theorem}{Theorems}
\crefname{innercustomcor}{Corollary}{Corollaries}
\crefname{innercustomprop}{Proposition}{Propositions}
\crefname{innercustomconj}{Conjecture}{Conjectures}

\crefname{equation}{Equation}{Equations}
\crefname{section}{Section}{Sections}
\crefname{subsection}{Subsection}{Subsections}
\crefname{figure}{Figure}{Figures}
\crefname{table}{Table}{Tables}
\crefname{enumi}{Item}{Items}

\newcommand{\Gr}{\mathrm{Gr}}

\newcommand{\PP}{\mathbf{P}}

\newcommand{\mfm}{\mathfrak{m}}

\newcommand{\FF}{\mathscr{F}}
\newcommand{\OO}{\mathscr{O}}

\newcommand{\GG}{\mathbf{G}}

\newcommand{\XX}{\mathcal{X}}

\newcommand{\bfA}{\mathbf{A}}
\newcommand{\bfC}{\mathbb{C}}
\newcommand{\bfk}{\mathbb{K}}
\newcommand{\QQ}{\mathbb{Q}}

\DeclareMathOperator{\Proj}{Proj}
\DeclareMathOperator{\Spec}{Spec}

\begin{document}
%-------------------------------------------------------------------------------
\author[R.\,Ramkumar]{Ritvik~Ramkumar}
\address{Department of Mathematics\\Notre Dame\\South Bend, IN\\USA}
\email{rramkuma@nd.edu}
%-------------------------------------------------------------------------------
\subjclass[2020]{}
%-------------------------------------------------------------------------------
\title{Hodge theory and Gr\"obner degenerations}

%-------------------------------------------------------------------------------
\begin{abstract} 
In this paper, we investigate the behavior of mixed Hodge structures under square-free Gr\"obner degenerations. We show that the weight-zero part of singular cohomology and the zeroth graded piece of the Hodge filtration are preserved under such degenerations,
complementing the homological rigidity established in the seminal work of Conca and Varbaro. 
We use these results to settle their conjecture concerning the singularities of homogeneous coordinate rings admitting such degenerations.
\end{abstract}

\maketitle

\section{introduction}

Let $S = \bfk[x_1,\dots,x_n]$ be a standard graded polynomial ring over a field $\bfk$ of characteristic $0$, and let $I \subseteq S$ be a homogeneous ideal. 
A term order on $S$ determines an initial ideal $J = \mathrm{in}(I)$ and a flat projective family $\XX \to \bfA^1$ with special fiber $\XX_0 = \Proj(S/J)$ and a trivialization 
$$
\XX|_{\GG_m} \cong \Proj(S/I) \times \GG_m.
$$

\begin{definition}The family $\XX \to \bfA^1$ is called a \textbf{Gr\"obner degeneration (over $\bfk$)}. If $J$ is radical, then we call the family a \textbf{square-free Gr\"obner degeneration}.
\end{definition}

Square-free Gr\"obner degenerations are fundamental in commutative algebra and algebraic geometry. 
Their usefulness stems in part from the fact that their special fibers are Stanley-Reisner schemes associated to simplicial complexes, allowing geometric and algebraic questions to be studied through combinatorial methods, while also providing geometric interpretations of combinatorial results.
A celebrated example is the work of Knutson and Miller, who constructed square-free Gr\"obner degenerations of matrix Schubert varieties with irreducible components of the special fiber indexed by reduced pipe dreams \cite{KM05}. 
These degenerations relate the multidegrees and $K$-polynomials of matrix Schubert varieties to pipe-dream combinatorics, giving geometric explanations of the Schubert and Grothendieck polynomials.
In their seminal work, Conca and Varbaro \cite{Conca_Varbaro} proved that square-free Gr\"obner degenerations also preserve the graded Hilbert functions of all local cohomology modules, and hence Cohen-Macaulayness and regularity.
In terms of sheaf cohomology, they proved

\begin{thm}[\cite{Conca_Varbaro}] \label{CV_thm} Let $\XX \to \bfA^1$ be a square-free Gr\"obner degeneration. Then for all $i,j$ we have 
$$
\dim_\bfk H^i(\XX_1,\OO_{\XX_1}(j)) =\dim_\bfk H^i(\XX_0,\OO_{\XX_0}(j)). 
$$
\end{thm}

Beyond these invariants, Conca and Varbaro asked whether a smooth variety admitting a square-free Gr\"obner degeneration must have a Cohen-Macaulay homogeneous coordinate ring \cite[Question 3.17]{Conca_Varbaro} . 
This question was subsequently formulated as a conjecture in \cite{CDNV}, with further equivalent formulations appearing more recently in \cite{HTVW}. 
Before we state one of these formulations, we recall the following definition.

\begin{definition} A homogeneous ideal $I\subseteq S$ is said to be \textbf{Herzog} if it has a square-free initial ideal for some term order. 
A projective scheme is \textbf{Herzog} if it is isomorphic to $\Proj(S/I)$ for some Herzog ideal $I \subseteq S$.
\end{definition}

\begin{conj} \label{Varbaro_conj} If $X$ is a smooth integral Herzog scheme, then  $H^i(X,\OO_X) =0$ for all $i\geq 1$.
\end{conj}

The case when $X$ is a curve was recently proven in \cite{HTVW}; in particular, smooth integral Herzog curves are always rational. 
If the term order is assumed to be degree reverse lexicographic, then the conjecture was proven in \cite{CDNV}.

Our approach is to study mixed Hodge structures along square-free Gr\"obner degenerations. 
Let $W_\bullet$ be the weight filtration on singular cohomology and $F^{\bullet}$ the Hodge filtration. 
We write $\Gr_F^0 = F^0/F^1$ for the zeroth graded piece (see \Cref{Hodge_section}).
We prove that such degenerations preserve the weight-zero part of singular cohomology and the zeroth graded piece of the Hodge filtration.

\begin{thm} \label{zero_stable}
Let $\XX \to \bfA^1$ be a square-free Gr\"obner degeneration over $\bfC$. Then for all $i\geq 0$, we have canonical isomorphisms
$$
W_0H^i(\XX_0,\mathbb Q)\cong W_0H^i(\XX_1,\mathbb Q), \, \text{ and } \, \Gr_F^0H^i(\XX_0,\mathbb C) \cong \Gr_F^0H^i(\XX_1,\mathbb C).
$$
\end{thm}

As a consequence, we prove a stronger generalization of \Cref{Varbaro_conj} to varieties with rational singularities.

\begin{thm} \label{Herzog} If $X$ is Herzog over $\bfk$ with rational singularities, then $H^i(X,\OO_X) = 0$ for all $i \geq 1$.
\end{thm}

Note that the argument in \cite[Section~2]{Conca_Varbaro} shows that the equality in \Cref{CV_thm} holds between the special and general fibers of any flat graded one-parameter degeneration to a square-free monomial ideal.
However, we will see in \Cref{Counterexample_reduced_general} that the analogous extension of \Cref{Herzog} fails.
In particular, it is not clear how to adapt the techniques of \cite{Conca_Varbaro} to prove this theorem. 

The motivation for our approach arose from joint work  with Cid-Ruiz on the fiber-full scheme \cite{FiberFull,FiberFullLocal}, which parameterizes subschemes with prescribed cohomology dimensions in every twist. 
Square-free Gr\"obner degenerations naturally correspond to curves on these schemes. 
Since fiber-full schemes need not be proper, questions about their compactifications led us to consider mixed Hodge structures, motivated by their role in compactifications of moduli spaces of abelian varieties \cite{AMRT,FC90}. 
The key observation is that the triviality of a Gr\"obner degeneration over $\GG_m$ allows one to construct a specialization map compatible with mixed Hodge structures. This also leads to a more general statement for isotrivial degenerations (\Cref{isotrivial_positive}). 

As a further application, we obtain a criterion for rational singularities of homogeneous  coordinate rings in terms of their associated projective schemes. 
This extends \cite[Corollary 3.4]{CDNV} to arbitrary term orders.

\begin{thm} \label{singularity_theorem} 
Let $I \subseteq S = \bfk[x_1,\dots,x_n]$ be a homogeneous prime ideal that has a square-free initial ideal $J$. 
Then $S/I$ has rational singularities if and only if $\Proj(S/I)$ has rational singularities;
in this case, $S/I$ and $S/J$ are Cohen-Macaulay with negative $a$-invariant.
\end{thm}

%; while the extension of \Cref{singularity_theorem} seems to a bit more subtle. 
Finally, in \Cref{Examples} we work out a number of examples of mixed Hodge structures along square-free Gr\"obner degenerations.

\section*{Future Directions} 
There have been several recent advances in the study of higher Du Bois and higher rational singularities; see, for example, \cite{Laza_Friedman,MOPW23,MP20, JKSY22}. 
It would be interesting to extend our results to these higher singularities. 
As in \Cref{singularity_theorem}, the vanishing of certain cohomology groups governs the singularities of the homogeneous coordinate ring \cite[Section 7]{SVV}. 

\begin{question}  
Under what additional conditions on a smooth Herzog scheme $X \subseteq \PP^n$ do we have
$ H^i(X,\Omega_X^j(m))=0 $ for all $i,j,m > 0$?
\end{question}
Note that such a vanishing statement cannot hold for every smooth Herzog scheme. 
Indeed, a smooth quadric hypersurface $Q \subseteq \PP^4$ is Herzog, but
$H^1(Q,\Omega_Q^2(1))\cong \bfk$ \cite[Section 4.1]{BTLM97}.

\section*{Acknowledgments} 
We thank Allen Knutson for discussions on Gr\"obner degenerations over the years. 
We thank Debaditya Raychaudhury for helpful conversations on Du Bois singularities.
We also thank David Eisenbud and  Alessio Sammartano for helpful comments. 
OpenAI Codex and Macaulay2 were used to compute Hodge structures on Stanley-Reisner schemes. 

\section*{Funding} Ritvik Ramkumar was  supported by NSF grant DMS-2548906.

\section{Mixed Hodge structures and Du Bois complexes} \label{Hodge_section}
In this section, we recall all the necessary background on mixed Hodge structures and Du Bois complexes. 
For projective Du Bois schemes, these give a natural identification of $H^i(X,\OO_X)$ with the zeroth Hodge quotient of singular cohomology. 
We begin with the classical theory.

\begin{notation}
Unless otherwise stated, all schemes are assumed to be separated and of finite type over $\bfC$. We write $X^{\mathrm{an}}$ for the associated complex analytic space, whose underlying set is $X(\bfC)$ with its analytic (Euclidean) topology. 
The groups $H^i(X,\QQ)$ and $H^i(X,\bfC)$ denote singular cohomology of this space, while $H^i(X^{\mathrm{an}},\mathscr G)$ denotes sheaf cohomology. 
We write $\QQ_{X^{\mathrm{an}}}$ and $\bfC_{X^{\mathrm{an}}}$ for the constant sheaves, whose sections are locally constant functions. 
For a coherent sheaf $\FF$, the notation $H^i(X,\FF)$ denotes Zariski sheaf cohomology.

For a bounded below complex of sheaves $\mathscr{K}^\bullet$ on $X$, we write $\mathbb H^i(X,\mathscr{K}^\bullet)$ for its hypercohomology in the Zariski topology.
We use the analogous notation on $X^{\mathrm{an}}$ for hypercohomology in the analytic topology.
\end{notation}

Since $X^{\mathrm{an}}$ is locally contractible and paracompact, the natural comparison is an isomorphism
\begin{equation}\label{constant_comp}
H^i(X,\QQ)\xrightarrow{\sim}H^i(X^{\mathrm{an}},\QQ_{X^{\mathrm{an}}});
\end{equation}
see \cite[Theorem III.1.1]{Bredon}.
The same statement holds with $\bfC$ coefficients.
If $X$ is projective and $\FF$ is coherent, GAGA \cite{SerreGAGA} gives the comparison isomorphism 
$$
H^i(X,\FF)\xrightarrow{\sim}H^i(X^{\mathrm{an}},\FF^{\mathrm{an}}).
$$
The inclusion $\bfC_{X^{\mathrm{an}}}\hookrightarrow\OO_{X^{\mathrm{an}}}$, of locally constant functions into holomorphic functions, 
induces, for projective $X$, a natural map via preceding isomorphisms
$$
H^i(X,\bfC)\longrightarrow H^i(X,\OO_X).
$$
For a smooth projective scheme, classical Hodge theory realizes $H^i(X,\OO_X)$ as a distinguished summand of $H^i(X,\bfC)$. 
Hodge structures record this decomposition and its compatibility with complex conjugation.

\begin{definition}
Let $H$ be a finite-dimensional vector space over $\QQ$, and put $H_\bfC=H\otimes_\QQ\bfC$. A \textbf{pure Hodge structure of weight $m$} on $H$ is a decomposition into complex subspaces, called the \textbf{Hodge decomposition},
$$
H_\bfC=\bigoplus_{p+q=m}H^{p,q}
$$
with $\overline{H^{p,q}}=H^{q,p}$. 
Note that conjugation is $\overline{h\otimes z}=h\otimes\overline{z}$. 
The \textbf{Hodge numbers} are $h^{p,q}=\dim_\bfC H^{p,q}$, and the \textbf{Hodge filtration} is
\begin{equation*} 
F^aH_\bfC :=\bigoplus_{p\geq a}H^{p,m-p}.
\end{equation*}
We write $\Gr_F^pH_\bfC=F^pH_\bfC/F^{p+1}H_\bfC$ for its associated graded pieces.
\end{definition}

\begin{remark} A finite decreasing filtration $F^\bullet$ defines a pure Hodge structure of weight $m$ if we have $H_\bfC=F^pH_\bfC\oplus\overline{F^{m-p+1}H_\bfC}$ for all $p$. 
In this case, we recover the Hodge decomposition by $H^{p,m-p}=F^pH_\bfC\cap\overline{F^{m-p}H_\bfC}$.
\end{remark}

We are now ready to state the classical Hodge decomposition. 
Given a scheme $X$, let  $\Omega_X^1$ denote the sheaf of K\"ahler differentials and $\Omega_X^p=\bigwedge^p\Omega_X^1$ the sheaf of differential $p$-forms.

\begin{thm}[ {\cite{DeligneII}}]\label{classical_Hodge}
If $X$ is smooth and projective, then $H^i(X,\QQ)$ has a natural pure Hodge structure of weight $i$, with
$$
H^i(X,\bfC)=\bigoplus_{p+q=i}H^{p,q}(X)
$$
where  $H^{p,q}(X)\cong H^q(X,\Omega_X^p)$. 
In particular, 
$$
\Gr_F^0H^i(X,\bfC)\cong H^i(X,\OO_X).
$$
\end{thm}

For singular projective schemes, cohomology in a fixed degree can have several weights. 
A mixed Hodge structure replaces the pure decomposition by a filtration whose successive quotients are pure.

\begin{definition}\label{def:mixed}
A \textbf{mixed Hodge structure} consists of a finite-dimensional rational vector space $H$, a finite increasing \textbf{weight filtration} $W_\bullet H$ by rational subspaces, and a finite decreasing \textbf{Hodge filtration} $F^\bullet H_\bfC$ by complex subspaces, such that
$$
\Gr_m^W H=W_mH/W_{m-1}H
$$
is a pure Hodge structure of weight $m$ for the filtration induced by $F^\bullet$. 
Explicitly,
\begin{equation*}\label{induced_hodge}
F^p(\Gr_m^WH)_\bfC =\frac{(F^pH_\bfC\cap (W_mH)_\bfC)+(W_{m-1}H)_\bfC}{(W_{m-1}H)_\bfC}.
\end{equation*}
\end{definition}

\begin{remark}
By finite we mean $W_mH=0$ for $m\ll0$ and $W_mH=H$ for $m\gg0$, while $F^pH_\bfC=H_\bfC$ for $p\ll0$ and $F^pH_\bfC=0$ for $p\gg0$. 
Just as before, we write $\Gr_F^pH_\bfC=F^pH_\bfC/F^{p+1}H_\bfC$. 
A pure structure of weight $m$ is viewed as mixed by setting $W_{m-1}H=0$ and $W_mH=H$.
\end{remark}

\begin{definition}\label{def:effective}
The \textbf{Hodge numbers} of a mixed Hodge structure $H$ are the Hodge numbers of its weight quotients, i.e., 
$$
h^{p,q}(H)=\dim_\bfC\Gr_F^p(\Gr_{p+q}^WH)_\bfC.
$$
We call $H$ \textbf{effective} if $h^{p,q}(H)=0$ unless $p,q\geq0$.
\end{definition}

\begin{definition}\label{Hodge_Tate_Def}
We use $\QQ(-r)$ to denote the one-dimensional pure Hodge structure of weight $2r$ and Hodge type
$(r,r)$.   

A mixed Hodge structure is of \textbf{Hodge-Tate type} if its odd weight quotients vanish and each $\Gr_{2r}^WH$ is a direct sum of copies of $\QQ(-r)$. 
Equivalently, $h^{p,q}(H)=0$ for $p\ne q$. 
\end{definition}

\begin{definition}
A \textbf{morphism} of mixed Hodge structures is a linear map $f:H\to H'$ such that
$$
f(W_mH)\subseteq W_mH' \, \text{ and } \, f_\bfC(F^pH_\bfC)\subseteq F^pH'_\bfC
$$
for all $m,p$. It is \textbf{strict} for the weight and Hodge filtrations, respectively, if
$$
f(W_mH) = f(H)\cap W_mH'  \, \text{ and } \, f_\bfC(F^pH_\bfC) = f_\bfC(H_\bfC)\cap F^pH'_\bfC.
$$
\end{definition}

In particular, if the morphism is strict, then the filtrations on the image induced from the source and target agree.

\begin{notation}
The weight filtration on $H$ induces a filtration on each Hodge quotient $\Gr_F^pH_\bfC$. In particular,
$$
W_m\Gr_F^pH_\bfC = \frac{(F^pH_\bfC\cap(W_mH)_\bfC)+F^{p+1}H_\bfC}{F^{p+1}H_\bfC}.
$$
\end{notation}

\begin{thm}[{\cite[Theorem 2.3.5]{DeligneII}}]
\label{Deligne_Lemma}
Mixed Hodge structures form an abelian category. Kernels and cokernels carry the induced and quotient filtrations, respectively.
Every morphism is strict for both filtrations. 
The functors $W_m$, $\Gr_m^W$, and $\Gr_F^p$ are exact.

Moreover, for every morphism $f:H\to H'$, the induced map
$$
\Gr_F^p(f_\bfC):
\Gr_F^pH_\bfC\longrightarrow\Gr_F^pH'_\bfC
$$
is strict for the induced weight filtrations. 
In particular, if this map is bijective, it is an isomorphism of filtered vector spaces.
\end{thm}

\begin{cor}\label{cor:strictness-consequences}
A bijective morphism of mixed Hodge structures is an isomorphism. 
A morphism between pure Hodge structures of different weights is zero.
\end{cor}

\begin{remark}\label{sub_Hodge} Given a mixed Hodge structure $H$, the induced filtrations make both $W_mH$ and $H/W_mH$ mixed Hodge structures. 
In particular, $F^p(W_mH)_\bfC=F^pH_\bfC\cap(W_mH)_\bfC$.
\end{remark}

Deligne extended the classical Hodge structure on smooth projective schemes to mixed Hodge structures on the cohomology of arbitrary complex schemes. 

\begin{thm}[{\cite[Proposition 8.2.2, Theorem 8.2.4]{DeligneIII}}] \label{Deligne_big_thm}
The groups $H^i(X,\QQ)$ carry functorial effective mixed Hodge structures. If $X$ is projective, then
$$
W_iH^i(X,\QQ)=H^i(X,\QQ).
$$
Thus, the only possible weights are between $0$ and $i$. If $X$ is also smooth, $H^i(X,\QQ)$ is pure of weight $i$.
\end{thm}

We now prove two simple lemmas illustrating what happens when the Hodge structure is effective. 
In particular, if $H$ is effective, we have $\Gr_F^0H_\bfC=H_\bfC/F^1H_\bfC$ and we can identify the weight-zero part.

\begin{lemma}\label{Eff}
Let $H$ be an effective mixed Hodge structure. Then $W_{-1}H=0$, $F^0H_\bfC=H_\bfC$, and $W_0H$ is pure of type $(0,0)$.
\end{lemma}
\begin{proof}
If $m<0$, then the Hodge numbers of $\Gr_m^WH$ vanish, and thus $\Gr_m^WH=0$. Since $W$ is finite, we have $W_{-1}H=0$. 
Each weight quotient has Hodge types $(p,q)$ with $p,q\geq0$, and hence $F^0(\Gr_m^W H)_\bfC=(\Gr_m^W H)_\bfC$.
By strictness, the short exact sequences
$$
0\longrightarrow W_{m-1}H \longrightarrow W_mH \longrightarrow \Gr_m^W H \longrightarrow0
$$
remain exact after complexifying and applying  $F^0$.
We now induct on $m$, beginning with $W_{-1}H=0$, to obtain $F^0(W_mH)_\bfC=(W_mH)_\bfC$.
Choosing $m\gg 0$ we obtain $F^0H_\bfC=H_\bfC$.
Finally, $W_0H=\Gr_0^W H$, whose only possible Hodge type is $(0,0)$.
\end{proof}

\begin{lemma}\label{Hodge_Tate_inclusion}
Let $H$ be an effective mixed Hodge structure. The quotient map $H_\bfC\to\Gr_F^0H_\bfC$ restricts to a functorial injection
$$
(W_0H)_\bfC \hookrightarrow\Gr_F^0H_\bfC
$$
with image $W_0\Gr_F^0H_\bfC$. 
If $H$ is of Hodge-Tate type, the above map is an isomorphism.
\end{lemma}
\begin{proof}
By \Cref{Eff}, we have $\Gr_F^0H_\bfC=H_\bfC/F^1H_\bfC$, and $W_0H$ is of type $(0,0)$.
Hence $(W_0H)_\bfC\cap F^1H_\bfC=F^1(W_0H)_\bfC=0$ and the map is injective. 
Its image is $W_0\Gr_F^0H_\bfC$ by the definition of the induced weight filtration.

If $H$ is Hodge-Tate, the nonzero weight quotients of $V=H/W_0H$ have type $(r,r)$ with $r>0$, so their zeroth Hodge quotients vanish. 
Exactness of $\Gr_F^0$, applied successively to the weight filtration of $V$, gives $\Gr_F^0V_\bfC=0$. Now applying it to $0\to W_0H\to H\to V\to0$ proves surjectivity.
\end{proof}

To apply this observation to schemes, we need to identify the zeroth Hodge quotient with sheaf cohomology. 
The classical Hodge decomposition did this for smooth projective schemes \Cref{classical_Hodge}; the Du Bois complex will provide the extension we need.

\begin{remark} For a smooth scheme $X$, the de Rham complex is
$
\Omega_X^\bullet=[\OO_X\xrightarrow{d}\Omega_X^1\xrightarrow{d}\Omega_X^2\to\cdots],
$
with $\OO_X$ in degree zero. Its filtration by form degree is $F^p\Omega_X^\bullet=\Omega_X^{\geq p}$, and thus its zeroth graded piece is $\OO_X$.
\end{remark}

More generally, for every reduced scheme $X$, one can construct a filtered complex $(\underline\Omega_X^\bullet,F)$ called the \textbf{Du Bois complex}; this agrees with the de Rham complex when $X$ is smooth \cite{DuBois}. 
Its zeroth graded piece $\underline\Omega_X^0:=F^0\underline\Omega_X^\bullet/F^1\underline\Omega_X^\bullet$ is, in general, a complex of $\OO_X$-modules, and there is a natural morphism
\begin{equation}\label{struc_map}
\OO_X\longrightarrow \underline\Omega_X^0.
\end{equation}

\begin{definition}\label{def:dubois}
A reduced scheme $X$ has \textbf{Du Bois singularities} if \eqref{struc_map} is a quasi-isomorphism.
\end{definition}

Smooth schemes are Du Bois because $\underline\Omega_X^0\cong\OO_X$ and schemes with rational singularities are Du Bois \cite[Corollary 5.5]{KovacsSchwedeSurvey}. 
We will see in the next section that Stanley-Reisner schemes are also Du Bois (\Cref{DuBois}). 
We do not need the general construction of the Du Bois complex; instead, we describe in detail a special case that applies to Stanley-Reisner schemes.

\begin{example}\label{smooth_intersect}
Let $X$ be a reduced scheme and assume there exists a collection $\{X_i\}_i$ of closed subschemes such that  $X=X_1\cup\cdots\cup X_s$ and the scheme-theoretic intersections
$
X_A:=\bigcap_{a\in A}X_a
$
are smooth for every non-empty $A \subseteq \{1,\dots,s\}$.

Let $i_A:X_A\hookrightarrow X$ be the inclusion. By \cite[Section 4.1 and Proposition 4.10]{DuBois}, the Du Bois complex is represented by the total complex of
$$
\mathscr D^{r,q} = \bigoplus_{|A|=r+1}(i_A)_*\Omega_{X_A}^q
$$
for $r,q \geq 0$.
The horizontal maps $\delta$ are alternating sums of restrictions to intersections (c.f. \Cref{interesection_complex_cc}), and the vertical maps are the de Rham differentials $d$.
Thus the total complex has
$
\operatorname{Tot}(\mathscr D)^n = \bigoplus_{r+q=n}\mathscr D^{r,q},
$
with differential $\delta+(-1)^r d$ on $\mathscr D^{r,q}$.
The filtration $F^p$ consists of forms of degree at least $p$.
Taking the zeroth graded piece retains only the row $q=0$,
so
$$
\underline\Omega_X^0
\cong
\left[
\bigoplus_{|A|=1}(i_A)_*\OO_{X_A} \xrightarrow{\delta}
\bigoplus_{|A|=2}(i_A)_*\OO_{X_A} \xrightarrow{\delta}
\bigoplus_{|A|=3}(i_A)_*\OO_{X_A} \longrightarrow\cdots
\right].
$$
The natural morphism $\OO_X\to\underline\Omega_X^0$ is the augmentation $f\mapsto(f|_{X_a})_a$.
In particular, $X$ is Du Bois if and only if the augmented complex
$$
0\longrightarrow\OO_X \longrightarrow \bigoplus_{|A|=1}(i_A)_*\OO_{X_A} \longrightarrow \bigoplus_{|A|=2}(i_A)_*\OO_{X_A} \longrightarrow\cdots
$$
is exact.
\end{example}

\begin{prop}[{\cite[Theorem 4.5(ii)-(iii)]{DuBois}}] \label{non_proper_map}
If $X$ is Du Bois, there is a functorial map
$$
\alpha_X:H^i(X,\bfC)\longrightarrow H^i(X,\OO_X).
$$
If $X$ is projective, this map factors through the Hodge quotient and induces a functorial isomorphism
\begin{equation}\label{eq:db-proper}
\Gr_F^0H^i(X,\bfC)\xrightarrow{\sim}H^i(X,\OO_X).
\end{equation}
\end{prop}

\begin{example} 
Assume that the scheme $X$ in \Cref{smooth_intersect} is Du Bois, or equivalently, that the augmented complex of structure sheaves is exact.
Consider the projection
$
\pi:\operatorname{Tot}(\mathscr D) \longrightarrow\mathscr D^{\bullet,0}
$
that is the identity on forms of degree zero and is zero on higher degree forms. 
Then $\alpha_X$ is the composition
$$
H^i(X,\bfC) \xrightarrow{\sim} \mathbb H^i(X,\operatorname{Tot}(\mathscr D))
\xrightarrow{\mathbb H^i(\pi)} \mathbb H^i(X,\mathscr D^{\bullet,0})
\xrightarrow{\sim} H^i(X,\OO_X).
$$
The first map is the comparison isomorphism in \cite[Theorem 4.5(ii)]{DuBois}, and the last is the inverse of the isomorphism induced by the augmentation.
\end{example}

\section{Stanley-Reisner schemes}
In this section we compute the weight-graded pieces of the cohomology of Stanley-Reisner schemes in terms of their component intersections. 
In particular, we show that their zeroth Hodge quotient is precisely the complexification of their weight-zero part.

\begin{definition} Let $\Delta$ be a simplicial complex on $[n]=\{1,\ldots,n\}$, with facets $\sigma_1,\ldots,\sigma_s$. 
For $\tau \subseteq [n]$, we write $\boldsymbol{x}^\tau=\prod_{j\in\tau}x_j$.
We denote the \textbf{Stanley-Reisner ideal} associated to $\Delta$ by $I_{\Delta} = (\boldsymbol{x}^{\tau}: \tau \notin \Delta) \subseteq S$ and the \textbf{Stanley-Reisner scheme} by  $Y_{\Delta} := \Proj(S/I_{\Delta})$. 
\end{definition}

\begin{notation} For a simplicial complex $\Delta$ we always use  $\sigma_1,\dots,\sigma_s$ to denote its facets and write $\Delta = \langle \sigma_1,\dots,\sigma_s \rangle$. 

For any nonempty subset $A \subseteq [s]:=\{1,\dots,s\}$ define $ \sigma_A = \bigcap_{a \in A} \sigma_a$ and $Y_A  = V(x_j: j\notin \sigma_A) \subseteq \PP^{n-1}$.
The schemes $Y_{\{a\}}$ are the irreducible components of $Y_\Delta$, and $Y_A=\bigcap_{a\in A}Y_{\{a\}}$. 
If $\sigma_A\ne\emptyset$, then $Y_A\cong\PP^{|\sigma_A|-1}$; otherwise $Y_A$ is empty.
\end{notation}

\begin{remark} \label{push}
Let $i_A:Y_A\hookrightarrow Y_\Delta$ be the inclusion, and use the same symbol for its analytification.  
Since this is a closed immersion, the comparison map \eqref{constant_comp} gives
$$
H^q(Y_\Delta^{\mathrm{an}},(i_A)_*\QQ_{Y_A^{\mathrm{an}}}) =H^q(Y_A^{\mathrm{an}},\QQ_{Y_A^{\mathrm{an}}}) \cong H^q(Y_A,\QQ).
$$
\end{remark}

\begin{example} \label{Hodge_plane}
For non-empty $Y_A$, its cohomology ring is
$$
H^{\bullet}(Y_A,\QQ) = \QQ[h_A]/(h_A^{|\sigma_{A}|})
$$
with $h_A = c_1(\OO_{Y_{A}}(1)) \in H^2(Y_A,\QQ)$ of Hodge type $(1,1)$. 
In particular, the only non-zero groups are  $H^{2r}(Y_A,\QQ)= \QQ(-r)  $ with $ 0 \leq r < |\sigma_A|$.

It follows that for any $A \subseteq B$, the pullback along $Y_B \hookrightarrow Y_A$ sends $h^r_A$ to $h_B^r$ whenever the latter exists, and to zero otherwise. 
In particular, after fixing the basis, the nonzero restriction maps are just the identity.
\end{example}

The higher nerves \cite{HigherNerves} record which component intersections contribute in each cohomological degree.

\begin{definition} Let $\Delta$ be a simplicial complex with facets $\sigma_1,\dots,\sigma_s$. 
For $j\geq 1$, the $j$-th \textbf{higher nerve} $N_j(\Delta) $ is the simplicial complex whose nonempty faces are $\{A \subseteq [s]: |\sigma_A| \geq j \}$.
\end{definition}

\begin{remark} The nerve $N_1(\Delta)$ records the nonempty intersections of the components of $Y_\Delta$. More generally, for $r \geq 0$, $N_{r+1}(\Delta)$ records those intersections of dimension at least $r$. 
Since each nonempty intersection $Y_A$ is a projective space, this is equivalent to having $H^{2r}(Y_A,\mathbb Q)\ne 0$.
\end{remark}

\begin{definition} \label{interesection_complex_cc}
The \textbf{intersection complex} $\mathscr C_{\Delta}^\bullet$ on $Y_{\Delta}^\text{an}$ is
$$
\begin{aligned}
0
\longrightarrow\bigoplus_{|A|=1} (i_A)_{\star}\QQ_{Y_A^{\text{an}}}
\longrightarrow\bigoplus_{|A|=2}(i_A)_{\star}\QQ_{Y_A^{\text{an}}}
\longrightarrow\cdots.
\end{aligned}
$$
The first direct sum is in degree $0$.

For $A\subseteq B$, let $\rho_{A,B}:(i_A)_{\star}\QQ_{Y_A^{\text{an}}} \to (i_B)_{\star}\QQ_{Y_B^{\text{an}}}$ denote the natural restriction map. 
The differentials of the complex are alternating sums of these maps; given $B=\{b_0<\cdots<b_{p+1}\}$ and a section $\boldsymbol{u}=(u_A)_{|A|=p+1}$ on any open set, we define
$$
(\delta \boldsymbol{u})_B =\sum_{i=0}^{p+1}(-1)^i \rho_{B-\{b_i\},B}(u_{B-\{b_i\}}).
$$
\end{definition}

\newcommand{\bfu}{\boldsymbol{u}}

It turns out that this complex resolves the constant sheaf on $Y_{\Delta}$.

\begin{lemma}\label{complexes_exact}
The natural map
$\mathbb Q_{Y_\Delta^{\mathrm{an}}}
\longrightarrow \mathscr C_\Delta^\bullet$
is a quasi-isomorphism. Moreover, the analogous complex
\begin{equation}\label{comp_res}
0 \longrightarrow \mathscr O_{Y_\Delta} \longrightarrow
\bigoplus_{|A|=1}(i_A)_*\mathscr O_{Y_A}
\longrightarrow \bigoplus_{|A|=2}(i_A)_*\mathscr O_{Y_A}
\longrightarrow \cdots
\end{equation}
is exact, where the differentials are alternating sums of
restriction maps.
\end{lemma}

\begin{proof}
The first assertion follows from
\cite[Example 5.3.7]{DeligneIII}. To see the second, let $\mathfrak{p}_A=(x_j:j\notin\sigma_A)$ and consider the complex
$$
0 \longrightarrow S/I_\Delta \longrightarrow \bigoplus_{|A|=1}S/\mathfrak{p}_A \longrightarrow \bigoplus_{|A|=2}S/\mathfrak{p}_A \longrightarrow \cdots.
$$
For each multidegree $\bfu\in\mathbb N^n$, let $\mathrm{supp}(\boldsymbol{u}) = \{j:u_j \ne 0 \}$. 
Then note that $(S/\mathfrak{p}_A)_{\bfu} \ne 0$ iff $\mathrm{supp}(\bfu) \subseteq \sigma_A$. 
In particular, in multidegree $\bfu$, the complex is either zero or the augmented cochain complex of the full simplex on $\{a:\mathrm{supp}(\bfu)\subseteq\sigma_a\}$.
Thus, the complex and its sheafification are exact.
\end{proof}

\begin{remark} \label{DuBois}
Every nonempty $Y_A$ is a projective space. By \Cref{smooth_intersect}, the complex \eqref{comp_res}, with its initial term $\OO_{Y_\Delta}$
omitted, represents $\underline{\Omega}_{Y_\Delta}^0$.
The exactness of the augmented complex \eqref{comp_res} shows that the natural map $\OO_{Y_\Delta}\to\underline{\Omega}_{Y_\Delta}^0$ is a quasi-isomorphism; in particular, $Y_\Delta$ has Du Bois singularities.
\end{remark}

We now use this resolution to compute the weight-graded pieces. We regard simplicial cohomology as a pure Hodge structure of type $(0,0)$.

\begin{prop}\label{Nerve_Gr} Let $I \subseteq S$ be a square-free monomial ideal and $Y_{\Delta}$ the associated Stanley-Reisner scheme. Then for all $i,r \geq 0$ we have
$$
\Gr_{2r}^W H^i(Y_{\Delta},\QQ)
\cong H^{i-2r}(N_{r+1}(\Delta),\QQ)
       \otimes_\QQ\QQ(-r) \quad
\text{ and } \quad
\Gr_{2r+1}^W H^i(Y_{\Delta},\QQ)=0.
$$\end{prop}

\begin{proof}
The constant-sheaf resolution gives the spectral sequence
\begin{equation*} \label{E1}
E_1^{p,q} =H^q(Y_\Delta^{\mathrm{an}},\mathscr C_\Delta^p) \cong\bigoplus_{|A|=p+1}H^q(Y_A,\QQ)
\quad\Longrightarrow\quad H^{p+q}(Y_\Delta,\QQ).
\end{equation*}
The first expression is sheaf cohomology, while the second and the abutment are singular cohomology, identified by \Cref{push} and \eqref{constant_comp}. 
By \cite[Theorem 8.1.15 and Proposition 8.3.5]{DeligneIII}, this is a spectral sequence of mixed Hodge structures, with an abutment filtration by mixed Hodge substructures. 

By \Cref{Hodge_plane}, the odd rows vanish,  and the row $q=2r$ is the simplicial cochain complex of $N_{r+1}(\Delta)$, with each copy of $\QQ$ replaced by $\QQ(-r)$. Therefore
\begin{equation}\label{E2}
E_2^{p,2r}=H^p(N_{r+1}(\Delta),\QQ)\otimes_\QQ\QQ(-r),
\qquad E_2^{p,2r+1}=0.
\end{equation}
Now every $E_2^{p,q}$ is pure of weight $q$. The differential $d_2$ has target of weight $q-1$, so it is zero. Hence $E_3=E_2$. Inductively, if $E_\ell=E_2$, the source and target of
$$
d_\ell:E_\ell^{p,q}\longrightarrow E_\ell^{p+\ell,q-\ell+1}
$$
have different weights for $\ell\geq2$; strictness makes this differential zero as well. This proves $E_2=E_\infty$.

Let $L^\bullet$ be the decreasing filtration
on $H^i(Y_\Delta,\QQ)$ induced by the spectral sequence; this is a filtration by mixed Hodge substructures.
The graded piece
$$
\Gr_L^pH^i(Y_\Delta,\QQ) \cong E_\infty^{p,i-p}=E_2^{p,i-p}
$$
is pure of weight $i-p$. It follows that the successive
quotients of $L^{i-m}H^i(Y_\Delta,\QQ)$ have weights at most $m$, while those of
$H^i(Y_\Delta,\QQ)/L^{i-m}H^i(Y_\Delta,\QQ)$ have weights greater than $m$. 
It follows from \Cref{Deligne_Lemma} that  $W_mH^i(Y_\Delta,\QQ)=L^{i-m}H^i(Y_\Delta,\QQ)$ and thus 
$$
\Gr_m^WH^i(Y_\Delta,\QQ) \cong E_2^{i-m,m}.
$$
Using \eqref{E2}, we obtain the desired result.
\end{proof}

\begin{cor}\label{Hodge_Tate_fiber}
The scheme $Y_\Delta$ is Du Bois, and $H^i(Y_\Delta,\QQ)$ is of Hodge-Tate type. Moreover,
$$
W_0H^i(Y_\Delta,\QQ)\cong H^i(\Delta,\QQ)\otimes_\QQ\QQ(0),
$$
and the natural comparison restricts to an isomorphism
\begin{equation}\label{eq:sr-zero-comparison}
(W_0H^i(Y_\Delta,\QQ))_\bfC\xrightarrow{\sim}H^i(Y_\Delta,\OO_{Y_\Delta}).
\end{equation}
\end{cor}
\begin{proof} 
Stanley-Reisner schemes are Du Bois by \Cref{DuBois} (also see \cite[proof of Corollary 4.2]{MustataPopa}). 
By \Cref{Nerve_Gr}, every even weight quotient is a direct sum of $\QQ(-r)$ and every odd quotient is zero; this is exactly the Hodge-Tate condition.
By \Cref{Deligne_big_thm} and \Cref{Eff}, $W_{-1}H^i(Y_\Delta,\QQ) = 0$ and thus, by \Cref{Nerve_Gr}, we have 
$$
W_0H^i(Y_{\Delta},\QQ) = \Gr_0^W H^i(Y_\Delta,\QQ) \cong H^i(N_1(\Delta),\QQ) \otimes_{\QQ}\QQ(0)
$$
By \cite[Theorem 3.3]{HigherNerves}, $N_1(\Delta)$ and $\Delta$ are homotopy equivalent.
Hence for all $i\geq 0$ we have
$$
H^i(N_1(\Delta),\QQ) \otimes_{\QQ}\QQ(0) \cong H^i(\Delta,\QQ) \otimes_{\QQ}\QQ(0).
$$
Finally, \Cref{Hodge_Tate_inclusion} identifies the complexification of this weight-zero piece with $\Gr_F^0H^i(Y_\Delta,\bfC)$, and the Du Bois comparison in \Cref{non_proper_map} identifies this quotient with $H^i(Y_\Delta,\OO_{Y_\Delta})$.
\end{proof}

\section{Degenerations and Hodge structures}
In this section we construct a morphism of mixed Hodge structures from the special fiber to the general fiber, under the assumption that the family is trivial away from $0$. By controlling the singularities of the total space, we prove our main results.

We begin by briefly describing mixed Hodge structures on relative singular cohomology.
For a closed subscheme $Z\subseteq X$, we use $H^i(X,Z;\QQ)$ to denote the singular cohomology of the pair $(X^{\mathrm{an}},Z^{\mathrm{an}})$.
These groups carry natural mixed Hodge structures, functorial with respect to morphisms of pairs, and the long exact sequence
$$
\cdots\longrightarrow H^i(X,Z;\QQ)
\longrightarrow H^i(X,\QQ)
\longrightarrow H^i(Z,\QQ)
\longrightarrow H^{i+1}(X,Z;\QQ)
\longrightarrow\cdots
$$
is an exact sequence of mixed Hodge structures
\cite[Example 8.3.8 and Proposition 8.3.9]{DeligneIII}.

\begin{prop} \label{functoriality} Let $f:\XX \to \bfA^1$ be a proper morphism such that $\XX|_{\GG_m} \cong \XX_1 \times \GG_m$ as $\GG_m$-schemes. Let $\iota_0: \XX_0 \to \XX$ and $\iota_1: \XX_1 \to \XX$ be the inclusions of the two fibers. For every $i$, the restriction is an isomorphism of mixed Hodge structures
$$
\iota_0^{\star}: H^i(\XX,\QQ) \overset{\cong}{\longrightarrow} H^i(\XX_0,\QQ) 
$$
This induces a morphism of mixed Hodge structures $s_i: H^i(\XX_0,\QQ) \to H^i(\XX_1,\QQ)$ where $s_i = (\iota_1^{\star})(\iota^{\star}_0)^{-1}$.
\end{prop}
\begin{proof}
Let $\Gamma$ be the closure of the graph of the
trivialization
$\XX|_{\GG_m}\cong\XX_1\times\GG_m$
inside the product
$\XX\times_{\bfA^1}(\XX_1\times\bfA^1)$.
We have a diagram
$$
\begin{tikzcd}[column sep=large, row sep=large]
\Gamma \arrow[r,"q"] \arrow[d,"p"'] &
\XX_1\times\bfA^1 \arrow[d,]\\
\XX \arrow[r,"f"'] & \bfA^1
\end{tikzcd}
$$
in which $p$ and $q$ are proper, and are isomorphisms over $\GG_m$. 
Let $\Gamma_0=p^{-1}(\XX_0) =q^{-1}(\XX_1\times\{0\})$.
By excision for relative cohomology (see \cite[Section 8.3.10]{DeligneIII}), we have isomorphisms of mixed Hodge structures 
$$
p^{\star}: H^i(\XX,\XX_0;\QQ) \overset{\cong}{\longrightarrow}  H^i(\Gamma,\Gamma_0;\QQ)
\, \text{ and } \,
q^{\star}:H^i(\XX_1\times\bfA^1,\XX_1\times\{0\};\QQ) \overset{\cong}{\longrightarrow}  H^i(\Gamma,\Gamma_0;\QQ).
$$

The long exact sequence in relative singular cohomology gives
$$
\begin{aligned}
H^{i-1}(\XX_1\times\bfA^1,\QQ) 
\longrightarrow H^{i-1}(\XX_1,\QQ)
\longrightarrow H^i(\XX_1\times\bfA^1,\XX_1;\QQ)
\longrightarrow H^i(\XX_1\times\bfA^1,\QQ)
\longrightarrow H^i(\XX_1,\QQ)
\end{aligned}
$$
where we identify $\XX_1 = \XX_1 \times \{0\}$.
The outer two maps are isomorphisms since the inclusion $\XX_1(\bfC) \times \{0\} \hookrightarrow \XX_1(\bfC) \times \bfC$ is a homotopy equivalence.
Thus, we must have $H^i(\XX_1\times\bfA^1,\XX_1\times \{0\};\QQ) = 0$. 
Since $p^\star$ and $q^\star$ are isomorphisms, this gives $H^i(\XX,\XX_0;\QQ)=0$ for every $i$.
Now applying the long exact sequence
to the pair $(\XX,\XX_0)$, we obtain
$$
0 = H^i(\XX,\XX_0;\QQ) \longrightarrow H^i(\XX,\QQ) \xrightarrow{\iota_0^\star}H^i(\XX_0,\QQ) \longrightarrow H^{i+1}(\XX,\XX_0;\QQ) =0 .
$$
Thus $\iota_0^{\star}$ is an isomorphism of $\QQ$-vector spaces. By functoriality, $\iota_0^{\star}$ is a morphism of mixed Hodge structures, and hence an isomorphism by \Cref{cor:strictness-consequences}.
Composing its inverse with the natural morphism $\iota_1^{\star}:H^i(\XX,\QQ) \to H^i(\XX_1,\QQ)$ gives the desired result.
\end{proof}

\begin{prop} \label{DuBois_Gr_isom} 
Let $f:\XX \to \bfA^1$ be a flat projective morphism such that $\XX|_{\GG_m} \cong \XX_1 \times \GG_m$ over $\GG_m$.  
If $\XX_0$ is Du Bois, then $\XX_1$ is Du Bois and, for every $i$, the map
$$
\Gr_F^0(s_i): \Gr_F^0H^i(\XX_0, \bfC) \xrightarrow{\sim} \Gr_F^0H^i(\XX_1,\bfC)
$$
is an isomorphism of vector spaces with their induced weight filtrations.
\end{prop}

\begin{proof}
By \cite[Theorem 4.1 and Corollary 4.2]{KovacsSchwede}, $\XX$ is Du Bois near $\XX_0$, and the fibers over an
open neighborhood of $0$ are Du Bois. 
By choosing a non-zero point in this neighborhood, the product trivialization implies that $\XX_1$, and hence every nonzero fiber, is Du Bois.
Since Du Bois singularities are preserved under smooth base change, $\XX|_{\GG_m}$ is also Du Bois.
Together with the neighborhood of $\XX_0$, this shows that the total space $\XX$ is Du Bois.

By  \cite[Theorem 4.6]{DuBois} and \Cref{non_proper_map}, we also get
$$
\dim_\bfC \Gr^0_F H^i(\XX_1,\bfC) = \dim_\bfC H^i(\XX_1,\OO_{\XX_1}) 
= \dim_\bfC H^i(\XX_0,\OO_{\XX_0}) = \dim_\bfC \Gr^0_F H^i(\XX_0,\bfC). 
$$
Since the dimensions are the same, it suffices to show that $\Gr^0_F(s_i)$ is surjective. 
By \Cref{functoriality} it suffices to prove that $\Gr^0_F(\iota_1^{\star})$ is surjective. 

Fix $i\geq0$ and let $V=H^i(\XX,\bfC)$ and $M=H^i(\XX,\OO_{\XX})$.
Since $f$ is projective over an affine base, $M$ is a finitely generated $\bfC[t]$-module. 
Since all the fibers of $f$ are Du Bois, \cite[Theorem 4.6]{DuBois} implies that $M$ is locally free and  $M/(t-a)M\cong H^i(\XX_a,\OO_{\XX_a})$ for all $a \in \bfC$.
By \Cref{non_proper_map}, we have a natural map $\alpha_{\XX}:V\to M$, which extends to a $\bfC[t]$-linear map $\Phi:V\otimes_\bfC\bfC[t]\to M$ given by  $v\otimes g\mapsto g\alpha_\XX(v)$.
For each $a$, functoriality and base change give the diagram
$$
\begin{tikzcd}[column sep=large,row sep=large]
V \arrow[r,"\iota_a^\star"]
  \arrow[d,"\alpha_{\XX}"'] &
H^i(\XX_a,\bfC) \arrow[d,"\alpha_{\XX_a}"]\\
M \arrow[r,"\mathrm{res}_a"'] &
M/(t-a)M \cong H^i(\XX_a,\OO_{\XX_a}).
\end{tikzcd}
$$
By \Cref{functoriality}, $\iota_0^\star$ is an isomorphism. Since $\XX_0$ is projective and Du Bois, \Cref{non_proper_map} implies that $\alpha_{\XX_0}$ is surjective. 
This implies that $\Phi$ is surjective modulo $(t)$.
Since the cokernel of $\Phi$ is finitely generated, Nakayama's lemma implies that $\Phi$ is surjective in an open neighborhood $U$ of $0$.
For any $a \in U\setminus \{0\}$, the commutative diagram now shows that $\alpha_{\XX_a}\circ\iota_a^\star$ is surjective. 
Under the product trivialization, the inclusions of the fibers over $a$ and $1$ are homotopic. 
Their pullbacks on singular cohomology therefore agree after identifying the fibers. 
Functoriality of $\alpha$ now shows that $\alpha_{\XX_1}\circ\iota_1^\star$ is surjective.
By \Cref{non_proper_map}, $\alpha_{\XX_1}$ induces an isomorphism
$$
\Gr_F^0H^i(\XX_1,\bfC) \cong H^i(\XX_1,\OO_{\XX_1}).
$$
Thus, $\Gr_F^0(\iota_1^\star)$ is surjective as required.
Finally, $\Gr_F^0(s_i)$ is strict for the induced weight filtrations by \Cref{Deligne_Lemma}, so its inverse preserves these filtrations.
\end{proof}

We are now ready to prove our main results.

\begin{proof}[Proof of \Cref{zero_stable}]
By \Cref{Hodge_Tate_fiber},  $\XX_0$ is Du Bois and thus \Cref{DuBois_Gr_isom} proves that $\Gr_F^0(s_i)$ is an isomorphism. 
Applying  \Cref{Hodge_Tate_inclusion} to this isomorphism we obtain
$$
(W_0H^i(\XX_1,\QQ))_{\bfC}  \cong W_0\Gr_F^0H^i(\XX_1,\bfC) \cong  W_0\Gr_F^0H^i(\XX_0,\bfC) \cong (W_0H^i(\XX_0,\QQ))_{\bfC}
$$ 
By functoriality, the complexification of $W_0(s_i)$ is an isomorphism. Hence $W_0(s_i)$ is an isomorphism over $\QQ$.
\end{proof}

\begin{proof}[Proof of \Cref{Herzog}] 
Let $\XX \to \bfA^1$ be a square-free Gr\"obner degeneration with $X = \XX_1$.
We first show that we may assume $\bfk = \bfC$.
Indeed, let $\bfk_0$ be the subfield of $\bfk$, generated over $\QQ$, by the coefficients of a Gr\"obner basis defining $\XX_1$. 
In particular, the Gr\"obner degeneration $\XX \to \bfA^1$ is then defined over $\bfk_0$.
Since $\bfk_0$ is finitely generated over $\QQ$, we may choose an embedding $\bfk_0 \hookrightarrow \bfC$.
Applying flat base change to a resolution, we see that rational singularities are unchanged under field extensions in characteristic $0$ \cite[Tag 02KH]{stacks_project}.
Thus, the model of $\XX_1$ over $\bfk_0$ and its base change to $\bfC$ have rational singularities.
By flat base change, it therefore suffices to prove the required vanishing over $\bfC$; it then holds over $\bfk_0$, and hence over $\bfk$.

Fix $i>0$ and note that by \Cref{Hodge_Tate_fiber} and  \Cref{DuBois_Gr_isom} we have the isomorphisms
$$
W_0 \Gr_F^0 H^i(\XX_0,\bfC)  =  \Gr_F^0H^i(\XX_0,\bfC) \xrightarrow{\sim}\Gr_F^0H^i(\XX_1,\bfC)
$$
compatible with the induced weight filtrations.  In particular, $\Gr_F^0 H^i(\XX_1,\bfC)$ is entirely of weight zero.

Choose a projective resolution of singularities $\pi:\widetilde{\XX_1} \to \XX_1$. 
By functoriality \cite[Section 3.21]{DuBois}, we have a commutative
diagram
$$
\begin{tikzcd}[column sep=large,row sep=large]
\Gr_F^0H^i(\XX_1,\bfC)
\arrow[r,"\Gr_F^0(\pi^*)"]
\arrow[d,']&
\Gr_F^0H^i(\widetilde{\XX_1},\bfC)
\arrow[d,]\\
H^i(\XX_1,\OO_{\XX_1})
\arrow[r,"\pi^*"']&
H^i(\widetilde{\XX_1},\OO_{\widetilde{\XX_1}}).
\end{tikzcd}
$$
The vertical maps are isomorphisms by \Cref{non_proper_map}, while the bottom map is an isomorphism since $\XX_1$ has rational singularities. 
Hence the top map is an isomorphism.
It preserves the induced weight filtrations because the underlying map on singular cohomology
is a morphism of mixed Hodge structures.
In particular,  $\Gr_F^0H^i(\widetilde{\XX_1},\bfC)$ must also be entirely of weight zero. 
However, since $\widetilde{\XX_1}$ is smooth and projective, \Cref{Deligne_big_thm} implies that  $H^i(\widetilde{\XX_1},\QQ)$ is pure of weight $i$, and thus the induced weight filtration on $\Gr_F^0H^i(\widetilde{\XX_1},\bfC)$ is concentrated in weight $i$. 
Since $i>0$, this implies that $\Gr_F^0H^i(\widetilde{\XX_1},\bfC)=0$, and chasing the diagram we obtain $H^i(\XX_1,\OO_{\XX_1})=0$, as required.
\end{proof}

\begin{proof}[Proof of \Cref{singularity_theorem}]
Let $\mfm = (x_1,\dots,x_n)$ denote the maximal ideal of $S$,  $X = \Proj(S/I)$ and $d= \dim(S/I)$. 
It is well known that if $S/I$ has rational singularities then so does $X$. 
Now assume that $X$ has rational singularities; in particular, the punctured spectrum $\Spec(S/I) - \{\mfm\}$ has rational singularities. 
If $d=1$, then $X = \Spec(\bfk)$ and $S/I$ is a polynomial ring in one variable. 
Thus, we may assume that $d \geq 2$. 

By \cite[Theorem 13.13]{Miller_Sturmfels} and \cite[Theorem 1.2]{Conca_Varbaro}, we have  $H^i_\mfm(S/I)_j =  H^i_\mfm(S/J)_j = 0$ for all $i$ and $j > 0$. 
Since the punctured spectrum of $S/I$ has rational singularities, $H^i_{\mfm}(S/I)$ has finite length for
every $i<d$ \cite[Remark 2.1]{CDNV}. By \cite[Theorem 1.2]{Conca_Varbaro}, the same holds for $H^i_{\mfm}(S/J)$.
Consequently, we have $H^i_\mfm(S/I)_j = H^i_\mfm(S/J)_j=0$ for $i<d$ and $j \ne0$ \cite[Remark 2.10]{Conca_Varbaro}.
By \Cref{Herzog}, $H^i_{\mfm}(S/I)_0 \cong H^{i-1}(X,\OO_X)=0$ for all $2 \leq i \leq d$. 
Since $I$ is prime, we have  $H^0_{\mfm}(S/I) = 0$ and an exact sequence
$$
0 \longrightarrow (S/I)_0 \longrightarrow H^0(X,\OO_X) \longrightarrow H^1_{\mfm}(S/I)_0 \longrightarrow 0.
$$
By \cite[Remark 1.3]{HTVW}, the first map is an isomorphism and we obtain $H^1_{\mfm}(S/I) = 0$.  
Thus, all local cohomology modules below degree $d$ vanish, and the top local cohomology vanishes in non-negative degrees.
In particular, both $S/I$ and $S/J$ are Cohen-Macaulay with negative $a$-invariant.

Since $S/I$ is Cohen-Macaulay and also normal away from $\mathfrak m$, which has height $d\geq2$, it
satisfies $(S_2)$ and $(R_1)$. In particular, $S/I$ is normal and by \cite[Theorem 5.1]{KSSW09} it follows that $S/I$ has rational singularities.
\end{proof}

We conclude this section by describing a more general setting in which our results hold.

\begin{definition} A flat projective morphism $\XX \to \bfA^1$ is called an \textbf{isotrivial degeneration (over $\bfk$)} iff $\XX|_{\GG_m} \cong \XX_1 \times \GG_m$ over $\GG_m$.
\end{definition}

\begin{definition} An isotrivial degeneration $\XX \to \bfA^1$ is said to be \textbf{square-free} if $\XX_0 \cong \Proj(R/J)$, where $R = \bfk[x_1,\dots,x_n]$ with $\deg(x_i) > 0$ and $J \subseteq R$ is a square-free monomial ideal.
\end{definition}

\begin{thm} \label{isotrivial_positive} Let $\XX \to \bfA^1$ be an isotrivial square-free degeneration. 
If $\bfk = \bfC$, then for all $i\geq 0$, we have canonical isomorphisms
$$
W_0H^i(\XX_0,\mathbb Q)\cong W_0H^i(\XX_1,\mathbb Q), \, \text{ and } \, \Gr_F^0H^i(\XX_0,\mathbb C) \cong \Gr_F^0H^i(\XX_1,\mathbb C).
$$
If $\XX_1$ has rational singularities, then for all $i>0$ we have 
$ 
H^i(\XX_1,\OO_{\XX_1})=0. 
$
\end{thm}
\begin{proof}
By descending the family, its trivialization over $\GG_m$, and the identification of its special fiber to a finitely generated subfield of $\bfk$, we reduce to $\bfk=\bfC$ as in the proof of \Cref{Herzog}.

We first show that the Stanley-Reisner scheme defined by a square-free monomial ideal is Du Bois for any positive grading. 
Let $R = \bfC[x_1,\dots,x_n]$ with $\deg(x_i) = a_i > 0$ and $S =  \bfC[y_1,\dots,y_n]$, with $\deg(y_i)=1$. 
Let $\mu_{m}$ denote the cyclic group of order $m$ and note that $G=\mu_{a_1} \times \mu_{a_2} \times \cdots \times \mu_{a_n}$ acts on $S$ by scaling. 
Let $I \subseteq R$ be a square-free monomial ideal and let $J\subseteq S$ be the square-free monomial ideal obtained from $I$ by replacing each $x_i$ with $y_i$. 
Then the morphism $R \longrightarrow S$ given by $x_i \mapsto y_i^{a_i}$ induces an isomorphism $R/I \cong (S/J)^G$. 
In particular, if we let $Y = \Proj(S/J)$ and $X = \Proj(R/I)$, we have $X = Y/G$.
Since $Y$ is Du Bois (\Cref{DuBois}), its finite group quotient $X=Y/G$ is also Du Bois \cite[Corollary 5.4]{KovacsSchwedeSurvey}.

By \cite[Chapter~III, Theorem~7.2]{Bredon1972}, pullback along $\pi:Y\to Y/G$ induces an isomorphism
$\pi^\star:H^i(Y/G,\QQ)\xrightarrow{\sim}H^i(Y,\QQ)^G$; this is an isomorphism of mixed Hodge structures \cite[Proposition~8.2.2]{DeligneIII}.
Since $J$ is a monomial ideal, the action of $G$ extends to an action of the torus $\mathbb{T} = (\bfC^{\star})^n$ on $\Proj(S/J)$. 
Since $\mathbb{T}$ is path-connected in the analytic topology, every element of $\mathbb{T}$ acts on $Y^{\mathrm{an}}$ by a map that is homotopic to the identity. 
Therefore, $\mathbb{T}$, and thus $G$, acts trivially on the singular cohomology group $H^i(Y,\QQ)$ \cite[pg. 201]{Hatcher}. In particular, 
$
H^i(X,\QQ)\cong H^i(Y,\QQ)^G \cong H^i(Y,\QQ).
$
In particular, $H^i(X,\QQ)$ is of Hodge-Tate type. 
Thus, the proofs of \Cref{zero_stable} and \Cref{Herzog} therefore apply to any square-free isotrivial degeneration.
\end{proof}

\begin{remark} Once can extend \Cref{isotrivial_positive} to any isotrivial degeneration $\XX \to \bfA^1$ such that $\XX_0$ is defined by a square-free monomial ideal in a positive multigrading (c.f. \cite[Remark 2.5]{Conca_Varbaro}).
\end{remark}

\section{Examples} \label{Examples}

In this section we will compute the Hodge structures along some well-known degenerations to illustrate our results. 
The term orders used below are lexicographic or degree reverse lexicographic, with variables ordered $x_0 > \cdots > x_n$.

We begin with a remark on the effect of normalization on mixed Hodge structures.

\begin{remark} \label{Kunneth}
Let $X$ be integral and $\pi:\widetilde{X}\to  X$ a finite normalization. 
Assume that $\pi$ is an isomorphism away from the closed points $p_1,\dots,p_k$ and let $n_i = |\pi^{-1}(p_i)|$ be the size of the fibers. 
In the analytic topology, we have an exact sequence of sheaves
$$
0 \longrightarrow \QQ_X \to \pi_{\star} \QQ_{\widetilde{X}} \longrightarrow \bigoplus_{i=1}^k \QQ_{p_i}^{\oplus(n_i-1)} \longrightarrow 0
$$
Since $\pi$ is finite, we have $H^i(X,\pi_{\star}\QQ_{\widetilde{X}}) = H^i(\widetilde{X},\QQ_{\widetilde{X}})$ and since $X$ and $\widetilde{X}$ are connected, we have $H^0(X,\QQ_{X}) = H^0(\widetilde{X},\QQ_{\widetilde{X}})$. 
The long exact sequence in sheaf cohomology gives us 
$$
0 \to \bigoplus_{i=1}^k \QQ(0)^{\oplus(n_i-1)} \to H^1(X,\QQ) \to H^1(\widetilde{X},\QQ) \to 0
$$
and $H^i(X,\QQ) \cong H^i(\widetilde{X},\QQ)$ for all $i\geq 2$. 
These maps are compatible with mixed Hodge structures.
\end{remark}

\begin{example} 
Let $\XX \to \bfA^1$ be a lexicographic degeneration of a twisted cubic in $\PP^3$. 
Explicitly, $\XX_1$ is defined by $(x_0x_2-x_1^2,x_0x_3-x_1x_2,x_1x_3-x_2^2)$  and $\XX_0$ is defined by $(x_0x_2,x_0x_3,x_1x_3)$.
The cohomology of the two fibers is as follows.
\begin{center}
\renewcommand{\arraystretch}{1.15}
\begin{tabular}{l|cc}
& $\XX_1$ (twisted cubic) & $\XX_0$\\
\hline
$H^0(-,\QQ)$ & $\QQ(0)$ & $\QQ(0)$\\
$H^1(-,\QQ)$ & $0$ & $0$\\
$H^2(-,\QQ)$ & $\QQ(-1)$ & $\QQ(-1)^{\oplus 3}$\\
\hline
\hline
$W_0H^{>0}$ & $0$ & $0$\\
$\operatorname{Gr}_F^0H^{>0}$ & $0$ & $0$\\
$H^{>0}(-,\OO)$ & $0$ & $0$\\
\end{tabular}
\end{center}
Since $\XX_1\cong\PP^1$, its cohomology is immediate. 
On the other hand, the Stanley-Reisner complex associated to $\XX_0$ is $\Delta=\langle\{0,1\},\{1,2\},\{2,3\}\rangle$. 
The complex and its nerves are as follows 
\medskip
\begin{center}
\begin{tikzpicture}[
    scale=0.9,
    vertex/.style={circle,fill=black,inner sep=1.5pt}
]
% Stanely-Reisner complex Delta
\begin{scope}
    \foreach \i in {0,1,2,3} {
        \coordinate (v\i) at (1.1*\i,0);
    }
    \draw[thick] (v0)--(v1)--(v2)--(v3);
    \foreach \i in {0,1,2,3} {
        \node[vertex,label=below:{$\i$}] at (v\i) {};
    }
    \node at (1.65,-0.95) {$\Delta$};
\end{scope}

% Ordinary nerve N_1(Delta)
\begin{scope}[xshift=5.2cm]
    \coordinate (a) at (0,0);
    \coordinate (b) at (1.3,0);
    \coordinate (c) at (2.6,0);

    \draw[thick] (a)--(b)--(c);

    \node[vertex,label=below:{$\sigma_1$}] at (a) {};
    \node[vertex,label=below:{$\sigma_2$}] at (b) {};
    \node[vertex,label=below:{$\sigma_3$}] at (c) {};

    \node at (1.3,-0.95) {$N_1(\Delta)$};
\end{scope}

% Higher nerve N_2(Delta)
\begin{scope}[xshift=9.8cm]
    \coordinate (a) at (0,0);
    \coordinate (b) at (1.3,0);
    \coordinate (c) at (2.6,0);

    \node[vertex,label=below:{$\sigma_1$}] at (a) {};
    \node[vertex,label=below:{$\sigma_2$}] at (b) {};
    \node[vertex,label=below:{$\sigma_3$}] at (c) {};

    \node at (1.3,-0.95) {$N_2(\Delta)$};
\end{scope}
\end{tikzpicture}
\end{center}
By  \Cref{Nerve_Gr}, its only nonzero cohomology groups are $H^0(\XX_0,\QQ)=\QQ(0)$ and $H^2(\XX_0,\QQ)=\QQ(-1)^{\oplus3}$.
The remaining entries follow from the Hodge-Tate property and \Cref{Hodge_Tate_inclusion}.
\end{example}

Although the cohomology of $\XX_0$  is always of Hodge-Tate type (\Cref{Hodge_Tate_fiber}), this need not hold for the general fiber $\XX_1$. 
A classical example of this phenomenon is a non-rational cubic threefold in $\PP^4$.

\begin{example} \label{Cubic_threefold}
Let $\XX\to \bfA^1$ be a degree reverse lexicographic
degeneration of a smooth cubic threefold in $\PP^4$.
Explicitly, $\XX_1$ is defined by $(x_0x_1x_2+x_3(x_0^2+x_1^2+x_2^2)+x_3^3+x_4^3)$, while
$\XX_0$ is defined by $(x_0x_1x_2)$.  We have the following cohomology table
\begin{center}
\renewcommand{\arraystretch}{1.2}
\begin{tabular}{l|cc}
& $\XX_1$ (cubic threefold)
& $\XX_0$ \\
\hline
$H^0(-,\QQ)$
& $\QQ(0)$ & $\QQ(0)$ \\
$H^1(-,\QQ)$
& $0$ & $0$ \\
$H^2(-,\QQ)$
& $\QQ(-1)$ & $\QQ(-1)$ \\
$H^3(-,\QQ)$
& $V$ & $0$ \\
$H^4(-,\QQ)$
& $\QQ(-2)$ & $\QQ(-2)$ \\
$H^5(-,\QQ)$
& $0$ & $\QQ(-2)$ \\
$H^6(-,\QQ)$
& $\QQ(-3)$ & $\QQ(-3)^3$ \\
\hline\hline
$W_0H^{>0}$ & $0$ & $0$ \\
$\Gr_F^0H^{>0}$ & $0$ & $0$ \\
{$H^{>0}(-,\mathscr{O})$} & $0$ & $0$ \\
\end{tabular}
\end{center}

The Hodge structure on a smooth cubic threefold is well known \cite[pg. 205]{HuybrechtsCubic}. 
In particular, $V = H^3(\XX_1,\QQ)$ is non-zero with $h^{2,1} = h^{1,2} = 5$ and $h^{3,0}=h^{0,3}=0$.  
On the other hand,  the Stanley-Reisner complex associated to $\XX_0$ is
$
\Delta=\langle
\{0,1,3,4\},\{0,2,3,4\},\{1,2,3,4\}
\rangle
$. The nerves are as follows

\begin{center}
\begin{tikzpicture}[
    scale=0.9,
    vertex/.style={circle,fill=black,inner sep=1.5pt}
]
\foreach \k/\name in {
    0/{N_1(\Delta)},
    1/{N_2(\Delta)},
    2/{N_3(\Delta)},
    3/{N_4(\Delta)}
}{
    \begin{scope}[xshift=\k*3.5cm]
        \coordinate (a) at (0,0);
        \coordinate (b) at (2,0);
        \coordinate (c) at (1,1.73);

        \ifnum\k<2
            \fill[black!40] (a)--(b)--(c)--cycle;
        \fi

        \ifnum\k<3
            \draw[thick] (a)--(b)--(c)--cycle;
        \fi

        \node[vertex,label=below:{$\sigma_1$}] at (a) {};
        \node[vertex,label=below:{$\sigma_2$}] at (b) {};
        \node[vertex,label=above:{$\sigma_3$}] at (c) {};

        \node at (1,-0.95) {$\name$};
    \end{scope}
}
\end{tikzpicture}
\end{center}
Once again, the special fiber column follows from \Cref{Nerve_Gr}. 
\end{example}

We now discuss two examples of Gr\"obner degenerations whose general fibers do not have rational singularities. 
These examples show that Cohen-Macaulayness and negativity of the $a$-invariant need not hold without the rational singularities hypothesis in \Cref{singularity_theorem}.

\begin{example} \label{Cycle_Hodge} Let $\XX \to \bfA^1$ be the Gr\"obner degeneration from \cite[Example~3.15]{Conca_Varbaro}. Explicitly, $\XX_1$ is defined by
$
I= (x_1x_4x_5-x_2x_4^2-x_3x_5^2,
x_2x_3-x_4x_5, x_1x_3x_4-x_3^2x_5-x_4^3,
x_1x_2x_5-x_2^2x_4-x_5^3)
$
and $\XX_0$ is defined by $(x_1x_4x_5,x_2x_3,x_1x_3x_4,x_1x_2x_5)$. The cohomology of the two fibers is as follows.
\begin{center}
\renewcommand{\arraystretch}{1.2}
\begin{tabular}{l|cc}
& $\XX_1$ (pinched scroll)
& $\XX_0$\\
\hline
$H^0(-,\QQ)$ & $\QQ(0)$ & $\QQ(0)$ \\
$H^1(-,\QQ)$ & $\QQ(0)$ & $\QQ(0)$ \\
$H^2(-,\QQ)$ & $\QQ(-1)^2$ & $\QQ(-1)$ \\
$H^3(-,\QQ)$ & $0$ & $0$ \\
$H^4(-,\QQ)$ & $\QQ(-2)$ & $\QQ(-2)^4$ \\
\hline
\hline
$W_0H^1$ & $\QQ$ & $\QQ$ \\
$\Gr_F^0H^1$ & $\bfC$ & $\bfC$ \\
{$H^1(-,\mathscr{O})$} & {$\bfC$} & {$\bfC$} \\
\end{tabular}
\end{center}
Geometrically, $\XX_1$ is the projection of a rational normal scroll of type $(2,2)$ from a point on a secant line. One can check that this projection is just the normalization $\pi:\PP^1 \times \PP^1 \to \XX_1$ given by
$$
([u_0:u_1],[v_0:v_1]) \mapsto [u_0^2v_1+ u_1^2v_0:u_0^2v_0:u_1^2v_1:u_0u_1v_1:u_0u_1v_0]. 
$$ 
Note that $\pi$ identifies the points $([1:0],[0:1])$ and $([0:1],[1:0])$ and is an isomorphism elsewhere. The general fiber column now follows from \Cref{Kunneth} and K\"unneth's formula. 

The Stanley-Reisner complex associated to $\XX_0$ is $\Delta = \langle \{1,2,4\}, \{1,3,5\},\{2,4,5\}, \{3,4,5\} \rangle$. The nerves are as follows

\begin{center}
\begin{tikzpicture}[
    scale=0.9,
    vertex/.style={circle,fill=black,inner sep=1.6pt}
]
% N_1: two filled triangles, with cycle
\begin{scope}
    \coordinate (a) at (-1.5,-0.3);
    \coordinate (b) at (1.7,0.1);
    \coordinate (c) at (0.5,2.2);
    \coordinate (d) at (0.3,0.65);

    \fill[black!40] (a)--(c)--(d)--cycle;
    \fill[black!40] (b)--(c)--(d)--cycle;

    \draw[thick] (a)--(b)--(c)--cycle;
    \draw[thick] (a)--(d)--(b);
    \draw[thick] (c)--(d);

    \node[vertex,label=left:{$\sigma_1$}] at (a) {};
    \node[vertex,label=right:{$\sigma_2$}] at (b) {};
    \node[vertex,label=above:{$\sigma_3$}] at (c) {};
    \node[vertex,label=below:{$\sigma_4$}] at (d) {};

    \node at (0,-1.0) {$N_1(\Delta)$};
\end{scope}

% N_2: path 1342
\begin{scope}[shift={(5,0)}]
    \coordinate (a2) at (0,0.65);
    \coordinate (b2) at (0.85,0.65);
    \coordinate (c2) at (1.70,0.65);
    \coordinate (d2) at (2.55,0.65);

    \draw[thick] (a2)--(b2)--(c2)--(d2);

    \node[vertex,label=below:{$\sigma_1$}] at (a2) {};
    \node[vertex,label=below:{$\sigma_3$}] at (b2) {};
    \node[vertex,label=below:{$\sigma_4$}] at (c2) {};
    \node[vertex,label=below:{$\sigma_2$}] at (d2) {};

    \node at (1.275,-1.0) {$N_2(\Delta)$};
\end{scope}

% N_3: four isolated vertices
\begin{scope}[shift={(9.4,0)}]
    \coordinate (a3) at (0,0.65);
    \coordinate (b3) at (0.85,0.65);
    \coordinate (c3) at (1.70,0.65);
    \coordinate (d3) at (2.55,0.65);

    \node[vertex,label=below:{$\sigma_1$}] at (a3) {};
    \node[vertex,label=below:{$\sigma_2$}] at (b3) {};
    \node[vertex,label=below:{$\sigma_3$}] at (c3) {};
    \node[vertex,label=below:{$\sigma_4$}] at (d3) {};

    \node at (1.275,-1.0) {$N_3(\Delta)$};
\end{scope}
\end{tikzpicture}
\end{center}
Here $N_1(\Delta)$ is the complete graph on $[4]$ along with the 2-simplices $\{1,3,4\}$ and $\{2,3,4\}$, so $H^1(N_1(\Delta),\QQ)\cong\QQ$. 
In particular, the nonzero weight-zero class in $H^1(\XX_0,\QQ)$ is preserved under the degeneration by \Cref{zero_stable}.
Since $\XX_1$ is integral but not locally Cohen-Macaulay \cite[Example~3.15]{Conca_Varbaro},
it does not have rational singularities, showing that the hypothesis of \Cref{singularity_theorem} cannot be omitted.
\end{example}

\begin{example}
Let $\XX \to \bfA^1$ be a lexicographic degeneration of a nodal cubic in $\PP^2$. Explicitly, let $\XX_1$ be defined by $(x_0x_1x_2+x_1^3+x_2^3)$
and $\XX_0$ be defined by $(x_0x_1x_2)$. The cohomology of the two fibers is as follows.
\begin{center}
\renewcommand{\arraystretch}{1.2}
\begin{tabular}{l|cc}
& $\XX_1$ (nodal cubic)
& $\XX_0$ \\
\hline
$H^0(-,\mathbb Q)$
& $\mathbb Q(0)$ & $\mathbb Q(0)$ \\
$H^1(-,\mathbb Q)$
& $\mathbb Q(0)$ & $\mathbb Q(0)$ \\
$H^2(-,\mathbb Q)$
& $\mathbb Q(-1)$ & $\mathbb Q(-1)^3$ \\
\hline
\hline
$W_0H^1$
& $\mathbb Q$ & $\mathbb Q$ \\
$\Gr_F^0 H^1$
& $\mathbb C$ & $\mathbb C$ \\
{$H^1(-,\mathscr{O})$} & {$\bfC$} & {$\bfC$} \\
\end{tabular}
\end{center}
The cohomology table of $\XX_1$ follows from applying \Cref{Kunneth} to the standard normalization $\PP^1 \to \XX_1$. 
On the other hand, the Stanley-Reisner complex corresponding to $\XX_0$ is $\Delta = \langle \{0,1\},\{0,2\},\{1,2\}\rangle$. The complex and its nerves are as follows

\begin{center}
\begin{tikzpicture}[
    scale=0.9,
    vertex/.style={circle,fill=black,inner sep=1.6pt}
]
% Delta: boundary of a triangle
\begin{scope}
    \coordinate (a) at (0,0);
    \coordinate (b) at (2,0);
    \coordinate (c) at (1,1.73);

    \draw[thick] (a)--(b)--(c)--cycle;

    \node[vertex,label=below:$0$] at (a) {};
    \node[vertex,label=below:$1$] at (b) {};
    \node[vertex,label=above:$2$] at (c) {};

    \node at (1,-0.9) {$\Delta$};
\end{scope}

% N_1: boundary of a triangle
\begin{scope}[shift={(4,0)}]
    \coordinate (a1) at (0,0);
    \coordinate (b1) at (2,0);
    \coordinate (c1) at (1,1.73);

    \draw[thick] (a1)--(b1)--(c1)--cycle;

    \node[vertex,label=below:{$\sigma_1$}] at (a1) {};
    \node[vertex,label=below:{$\sigma_2$}] at (b1) {};
    \node[vertex,label=above:{$\sigma_3$}] at (c1) {};

    \node at (1,-0.9) {$N_1(\Delta)$};
\end{scope}

% N_2: three isolated vertices
\begin{scope}[shift={(8,0)}]
    \coordinate (a2) at (0,0);
    \coordinate (b2) at (2,0);
    \coordinate (c2) at (1,1.73);

    \node[vertex,label=below:{$\sigma_1$}] at (a2) {};
    \node[vertex,label=below:{$\sigma_2$}] at (b2) {};
    \node[vertex,label=above:{$\sigma_3$}] at (c2) {};

    \node at (1,-0.9) {$N_2(\Delta)$};
\end{scope}
\end{tikzpicture}
\end{center}
In particular,
$
W_0H^1(\XX_0,\QQ) \cong H^1(\Delta,\QQ)\cong\QQ.
$
In contrast to the twisted cubic, $H^1(\XX_1,\QQ)$ is entirely of weight zero, so it survives in $\Gr_F^0$. Although $\XX_1$ is arithmetically Cohen-Macaulay, its $a$-invariant is zero.
\end{example}

Finally, here is an example showing that an arbitrary flat degeneration with reduced fibers does not preserve the weight-zero parts. In particular, the conclusion of \Cref{Herzog} need not hold without the Gr\"obner hypothesis (even if the total space $\XX$ is smooth).

\begin{example} \label{Counterexample_reduced_general}
Consider the flat family $\XX\to\bfA^1$ of plane
cubics defined by
$$
\XX=V\bigl(x_0x_1x_2+
t(x_0^3+x_1^3+x_2^3)\bigr)
\subseteq\PP^2\times\bfA^1.
$$
The special fiber $\XX_0$ is the same as the previous example, while $\XX_1$ is a smooth elliptic curve. The cohomology of the two fibers is as follows.
\begin{center}
\renewcommand{\arraystretch}{1.2}
\begin{tabular}{l|cc}
& $\XX_1$ (elliptic curve) & $\XX_0$\\
\hline
$H^0(-,\QQ)$ & $\QQ(0)$ & $\QQ(0)$\\
$H^1(-,\QQ)$ & $V$ & $\QQ(0)$\\
$H^2(-,\QQ)$ & $\QQ(-1)$ & $\QQ(-1)^3$\\
\hline
\hline
$W_0H^1$ & $0$ & $\QQ(0)$\\
$\operatorname{Gr}_F^0H^1$ & $\bfC$ & $\bfC$\\
$H^1(-,\OO)$ & $\bfC$ & $\bfC$\\
\end{tabular}
\end{center}
The cohomology of $\XX_0$ was already computed in the preceding example. 
Since $\XX_1$ is a smooth elliptic curve, $V:=H^1(\XX_1,\QQ)$ is pure of weight one with
$h^{1,0}=h^{0,1}=1$ \cite[Section 1.3]{HuybrechtsCubic}. In particular,
$W_0H^1(\XX_1,\QQ)=0$ and $W_0H^1(\XX_0,\QQ)\cong\QQ$. Thus, the weight-zero part of $H^1(-,\QQ)$ is not preserved along the degeneration. Since the family  $\XX \to \bfA^1$ has singular fibers over non-zero values (for example, $t=-\frac{1}{3}$), the family is not trivial over $\GG_m$.
\end{example}

\bibliographystyle{amsalpha} 
\bibliography{references}

\end{document}